\documentclass[12pt]{amsart}
\usepackage{color}
\usepackage{amssymb}
\usepackage{epsfig}
\usepackage{url}
\usepackage{setspace}
\theoremstyle{plain}

\newtheorem{thm}{Theorem}[section]

\newtheorem{rem}[thm]{Remark}
\newtheorem{ques}[thm]{Question}

\newtheorem{exam}[thm]{Example}
\newtheorem{prob}[thm]{Problem}
\def\cal{\mathcal}
\def\bbb{\mathbb}
\def\op{\operatorname}
\renewcommand{\phi}{\varphi}

\newcommand{\N}{\bbb{N}}
\newcommand{\Z}{\bbb{Z}}
\newcommand{\Q}{\bbb{Q}}

\begin{document}
% \title[A note on the diophantine equation $(ap^2+bq^2)^2=f(x)$]{A note on the diophantine equation $(ap^2+bq^2)^2=f(x)$ and related results}
\title[On representing coordinates of points on elliptic curves]{On representing coordinates of points on elliptic curves by quadratic forms}

\author{Andrew Bremner and Maciej Ulas}

\keywords{Diophantine equations, rational points, elliptic curves, representations, quadratic forms}
\subjclass[2010]{11D25, 11D41, 11E16, 11G05}
\thanks{The research of the second author was partially supported by the grant of the Polish National Science Centre no. UMO-2012/07/E/ST1/00185}

\maketitle
\begin{abstract}
Given an elliptic quartic of type  $Y^2=f(X)$ representing an elliptic curve of positive rank over $\Q$, we investigate the question of when the $Y$-coordinate can be represented by a quadratic form of type $ap^2+bq^2$. In particular, we give examples of equations of surfaces of type
$c_0+c_1x+c_2x^2+c_3x^3+c_4x^4=(ap^2+bq^2)^2$, $a,b,c \in \Q$ where we can deduce the existence of infinitely many rational points.  We also investigate surfaces of type $Y^2=f(a p^2+b q^2)$ where the polynomial $f$ is of degree $3$.
\end{abstract}

\section{Introduction}\label{sec1}

Cornelissen and Zahidi~\cite{CorZa} introduce the quartic
surface defined by the equation
\begin{equation}
\label{CZ}
V: \; (A^2+B^2)(A^2+11 B^2) = 225 (P^2-5 Q^2)^2,
\end{equation}
with interest in the rational points on $V$ in regard to undecidability questions
in mathematical logic.
There is an associated elliptic curve
\[ E: \; Y^2 = (X^2+1)(X^2+11), \]
and the existence of rational points on $V$ is equivalent to the existence of rational points
on $E$ whose $Y$-coordinate is representable by the quadratic form $15(P^2-5Q^2)$. The curve
$E$ has rational rank 1, with generator $(\frac{1}{2},\frac{15}{4})$,
so the group
$E(\Q)$ is explicitly known. But the problem of deciding whether the $Y$-coordinate of
a point can be represented by the form $15(P^2-5Q^2)$ demands knowledge of the factorization
of the $Y$-coordinate into primes, which becomes computationally intractable for a point given by
a large multiple of the generator. \\
Here, we investigate the question of when the $Y$-coordinate of a rational point on
an elliptic quartic can be represented by a quadratic form, and in particular give
several examples of equations of surfaces of type
\[ \cal{V} :\; c_0 + c_1 x + c_2 x^2 + c_3 x^3 + c_4 x^4 = (a p^2 + b q^2)^2, \qquad a,b,c_i \in \Q \]
where we can deduce the existence of infinitely many rational points. This is achieved by constructing a hyperelliptic quartic curve $\cal{C}'$ with map $\psi: \cal{C}' \rightarrow \cal{V}$. When $\cal{C}'$ is an elliptic curve of positive rank, the result will follow.\\
Second, we consider elliptic curves $E:\,Y^2=f(X)$, $f$ cubic, and use similar techniques to investigate when
there can be infinitely many rational points on $E$ with $X$-coordinate representable by a quadratic form $ap^2+bq^2$.
We illustrate the ideas with many examples. Throughout the paper, an important role is played by the use of machine computations, particularly using the computer algebra system Magma~\cite{Mag}.

%\section{Results}\label{sec2}
\section{First construction}\label{sec2}

Let $a, b\in\Z$, where we suppose that the quadratic form $ax_1^2+bx_2^2$ is irreducible, and let $f(X)=\sum_{i=0}^{4}c_{i}X^{i}\in\Q[x]$ be a non-singular quartic. We are interested in when the quartic curve
$$
\cal{C}:\;Y^2=f(X)
$$
can have infinitely many rational points with the $Y$-coordinate being representable by the quadratic
form $a x_1^2+b x_2^2$.
Consider the quartic surface given by the equation
$$
\cal{V}:\;(ap^2+bq^2)^2=f(X).
$$
Multiplying through by $a^2$, we may suppose without loss of generality that $a=1$.
Suppose $(p_0,q_0,X_0) \in \cal{V}(\Q)$. Then $f(X_0) \neq 0$, otherwise $p_0^2+b q_0^2=0$ and the form $p^2+b q^2$ is reducible. It is clear that there exist infinitely many points $(p,q,X_0) \in \cal{V}(\Q)$, because
the conic
%$a p_0^2+b q_0^2 = a p^2+b q^2$
$p_0^2+b q_0^2 = p^2+b q^2$ is rationally parameterizable. Further, from this observation and the fact that the curve $p^4=f(X)$ is of genus 3, it is also clear that we may assume if necessary that $p_{0}q_{0}\neq 0$. When we refer to
$\cal{V}(\Q)$ being infinite, we shall henceforth mean that $\pi_X(\cal{V}(\Q))$ is infinite,
where $\pi_X: \cal{V}(\Q) \rightarrow \Q$ is the projection onto the $X$-coordinate.
It is clear that a necessary condition for $\cal{V}$ to have infinitely many rational points is that
the set $\cal{C}(\Q)$ is infinite.

Following a suitable change of variable we can move the point $(p_0,q_0,X_0)$ to infinity and eliminate the coefficient of $X^{3}$ in $f$. We thus consider the quartic surface
\begin{equation*}
%\cal{V}:\;(ap^2+bq^2)^2=m^2X^4+a_{2}X^2+a_{1}X+a_{0}, \qquad m = a p_0^2+b q_0^2,
\cal{V}:\;(p^2+bq^2)^2=m^2X^4+a_{2}X^2+a_{1}X+a_{0}, \qquad m = p_0^2+b q_0^2.
\end{equation*}
In order to find more points on $\cal{V}$, write
\begin{equation}\label{sub1}
p=p_0 T,\quad q= q_0 T+\frac{U}{p_0},\quad X=T+V,
\end{equation}
where $U, V, T$ are to be determined. This substitution gives
$$
%(ap^2+bq^2)^2-f(X)=\sum_{i=0}^{3}C_{i}T^{i},
(p^2+bq^2)^2-f(X)=\sum_{i=0}^{3}C_{i}T^{i},
$$
where $C_{i}\in \Q[U,V]$ for $i=0,1,2,3$. In particular $C_{3}=-4m(mp_0V-b q_0 U)/p_0$. Taking $V=b q_0 U/mp_0$,
there results $F(T):=B_{2}T^2+B_{1}T+B_{0}=0$ where $B_{i}=B_{i}(U)$  is in general of degree $4-i$ for $i=0,1,2$.
Now a point of the form (\ref{sub1}) lies on $\cal{V}$ if and only if $F(T)=0$ has rational roots, which is equivalent to the existence of $U \in \Q$ such that the discriminant of $F$ with respect to $T$
%, say $\Delta$,
is a square in $\Q$. We thus consider the curve $\cal{C}':\; W^2/m^2 = B_1^2-4 B_0 B_2$, that is,
\begin{equation*}
\cal{C}':\;W^2=-8b^3U^6+4b^2a_2U^4+8ba_0m^2 U^2+(a_1^2-4 a_0 a_2) m^2=:G(U).
\end{equation*}
Here, $T=(-B_1+W/m)/(2 B_2)$.
For general coefficients $a_{0}, a_{1}, a_{2}$ and given $b$ the curve $\cal{C}'$ is hyperelliptic of genus 2. By Faltings Theorem, the set $\cal{C}'(\Q)$ is finite. Thus a necessary condition for $\cal{C}'$ to have infinitely many rational points is the vanishing of the discriminant (with respect to $U$) of the polynomial $G$. In this case the genus of $\cal{C}'$ is at most $1$ and there is a chance for $\cal{C}'$ to have infinitely many rational points.

\begin{rem}
\rm{
If the approach of this section is used on the surface (\ref{CZ}), then the corresponding curve $\mathcal{C}'$ is of genus 2.
}
\end{rem}

\noindent
We have
\begin{equation*}
\op{Disc}_{U}(G)=2^{21}b^{15}(a_1^2-4 a_0 a_2)m^2 \op{Disc}_{X}(f(X))^2,
\end{equation*}
and since $\op{Disc}_{X}(f(X)) \neq 0$, then $\op{Disc}_{U}(G)=0$ if and only if $a_1^2-4 a_0 a_2=0$.
Hence the polynomial $f$ takes the form
\begin{equation*}
f(X)=m^2 X^4 + c(d X+e)^2
\end{equation*}
for some $c, d, e \in\Q$ with $ce\neq 0$. If $d \neq 0$ then we can clearly assume that $d=1$.
The corresponding polynomial $G$ is divisible by $4U^2$ and the curve $\cal{C}'$ may
now be taken in the form
\begin{equation*}
\cal{C}':\; w^2=-b(2b^2U^4-b c d^2U^2-2 c e^2 m^2),
\end{equation*}
% -b*(2*b^2*p0^4*U^4-b*c*d^2*p0^2*U^2-2*c*e^2*m^2)
with $W=2Uw$. \\ \\
Tracing back the maps, we obtain the mapping
$$
\psi: \mathcal{C}'\ni (U,w)\mapsto (p,q,X)\in \mathcal{V}
$$
 given by
 \begin{align}\label{map1}
p&=\frac{p_0cdem-bq_0U(2b U^2-cd^2)+p_0Uw}{m(2b U^2-cd^2)},\notag\\
q&=\frac{cdeq_0m+p_0U(2b U^2-cd^2)+q_0Uw}{m(2b U^2-cd^2)}, \\
X&=\quad \frac{cdem+Uw}{m(2b U^2-cd^2)}. \notag
\end{align}
%  (p,q,X)=(p0*(c*d*e*m - b*q0*U*(-c*d^2 + 2*b*p0^2*U^2) + p0*U*w)/(m*(-c*d^2 + 2*b*p0^2*U^2)), (c*d*e*q0*m + p0^2*U*(-c*d^2 + 2*b*p0^2*U^2) + p0*q0*U*w)/(m*(-c*d^2 + 2*b*p0^2*U^2)), (p0*U*w+c*d*e*m)/(m*(2*b*p0^2*U^2-c*d^2)) )
The following theorem now follows:
%Thus infinitely many points on $\mathcal{C}'$ implies not only infinitely many points
%on $\mathcal{C}$, but also infinitely many points on $\mathcal{V}$.
\begin{thm}\label{firstthm}
Let $\mathcal{V}$ denote the surface $(p^2+b q^2)^2 = m^2 X^4 +c(d X+e)^2$, and let $\mathcal{C}'$ be the quartic curve $w^2 = -b(2b^2U^4-b c d^2U^2-2 c e^2 m^2)$. Then $\mathcal{C}'(\Q)$ infinite implies $\mathcal{V}(\Q)$ infinite.
\end{thm}
\begin{proof}
Assume that the set $\cal{C}'(\Q)$ is infinite. To get the result it is enough to prove that the image of the map $\psi$ constructed above is infinite in the set $\cal{V}(\Q)$. This is equivalent to the fact that if $(p',q',X')\in\cal{V}(\Q)$ and $Y=p'^2+bq'^2$ then
there are only finitely many rational points $(U,w)\in\cal{C}'$ such that $p^2+bq^2=Y$, where $p, q$ are given by (\ref{map1}). Eliminating $w$ between the equations
\[ Y = p^2+b q^2, \qquad w^2 = -b(2b^2U^4-b c d^2U^2-2 c e^2 m^2) \]
gives an equation of degree $12$ in $U$, which is not identically zero since the coefficient of $U^{12}$ equals $4b^6p_0^4(-1+2p_0^4)^2(p_0^2+b q_0^2)^2$. Hence there are only finitely many possibilities for $U$, as required.
% This is clearly the case as the following reasoning shows. The corresponding system of equations has common rational solution $(U,w)$ if and only if the resultant (with respect to $w$), say $\op{Res}$, of the numerator of the rational function $Y-p^2-bq^2$ and the polynomial $w^2+b(2b^2U^4-b c d^2U^2-2 c e^2 m^2)$ vanishes. A quick computation reveals that the polynomial $\op{Res}=\op{Res}(U)$ is of degree 12. If at least one coefficient of the polynomial $\op{Res}$ does not vanish we are done because in this case we have only finitely many possibilities for $U$. Suppose that $\op{Res}\equiv 0$ in $\Q[U]$ and consider vanishing of the leading coefficient of $\op{Res}$, i.e.
%\begin{equation*}
%4b^6[(b q_0^2+6 p_0^2-8 p_0+3)^2-8(p_0-1)^2(2p_0-1)^2]=0.
%\end{equation*}
%It is clear that this equation has not solutions in rational numbers with $b(p_0-1)(2p_0-1)\neq 0$ due to the assumption on $b$ and the fact that $\sqrt{2}\not\in\Q$. If $p_0=1$ then we get the equation $bq_0^2+1=0$ and thus $b=-\square$ and our form is reducible - the case we excluded. Similarly, if $p_0=1/2$ then the coefficient near $U^{10}$ takes the form $-Y/32q_0^{10}$ and thus $Y=0$ which implies that our form is reducible. Our theorem is proved.
%}
\end{proof}
\begin{rem}
{\rm
It has been pointed out to us that $\mathcal{C}'(\Q)$ infinite in the hypothesis
of the theorem may be weakened to $\mathcal{C}'(\Q) \neq \emptyset$ and $\mathcal{C}(\Q)$
infinite. For if $E'$ denotes the elliptic curve associated to $\mathcal{C}'$, then
the composition $\phi: \mathcal{C}' \rightarrow \mathcal{V} \rightarrow \mathcal{C}$
extends to a finite morphism $E' \rightarrow E$. If $\mathcal{C}'(\Q) \neq \emptyset$,
take a rational point in $\mathcal{C}'(\Q)$, corresponding to $O' \in E'(\Q)$, and
let its image under $\phi$ correspond to $O \in E(\Q)$.
Then $\phi$ induces an isogeny $(E',O') \rightarrow (E,O)$ of degree $2$. Thus
if $\mathcal{C}(\Q)$ is infinite, then either $\mathcal{C}'(\Q)$ is empty
or $\mathcal{C}'(\Q)$ is infinite.
}
\end{rem}

\noindent
It is straightforward to compute a cubic model $E$ for $\mathcal{C}$; $E$ takes the form
\[ E: \; y^2 = x(x^2 +c d^2 x-4 c e^2 m^2), \]
with map $\mathcal{C} \rightarrow E$ given by
\[ (x,y) = (-2m(Y-mX^2), \; -2m(2 m^2 X^3-2m X Y+c d^2 X+c d e)). \]
The curve $\mathcal{C}'$ occurs as a homogenous space in a standard $2$-descent on $E$,
on setting $x$ equal either to $-2b$ or $2 b c$ in $\Q^*/\Q^{*2}$.  It follows
that $\mathcal{C}'(\Q) \neq \emptyset$ provided that there exists a point
$(X_0,Y_0) \in \mathcal{C}(\Q)$ with $Y_0-X_0^2 \equiv b \delta \bmod{\Q^{*2}}$,
where $\delta=m \text{ or } -mc$.

\begin{exam}
{\rm
Let $(p_0,q_0)=(1,0)$ and $(b,c,d,e)=(1,\ell,0,1)$ where $\ell$ is prime, so that $\cal{C}:\;Y^2=X^4+\ell$ and $\cal{C}':\;w^2=2(\ell-U^4)$.
Assume the rank of $\cal{C}$ is positive, with $(X_0,Y_0)$ a point of infinite order. Without loss of
generality, on changing the sign of $Y_0$ if necessary, $Y_0-X_0^2=\alpha^2$ and $Y_0+X_0^2=\ell/\alpha^2$. Then
$(U,w)=(\alpha,2\alpha X_0)$ is a point on $\cal{C}'$ of infinite order. Accordingly, the set of rational points on the surface
$$
\cal{V}:\;(p^2+q^2)^2=X^4+\ell
$$
is infinite. The above mapping reduces to
\[ (U,w) \rightarrow (p,q,X)=\left(\frac{w}{2U}, U, \frac{w}{2U}\right). \]
}
\end{exam}

\begin{exam}
{\rm
Set $D=m^2(m^2+2n^2)$ with point $P=(n,m^2+n^2)$ of infinite order on
the curve $\cal{C}: Y^2=X^4+D$. Take $b=1$, so that $\cal{C}': w^2=2(D-U^4)$
with point of infinite order $P'=(m, 2m n)$. Then the set of rational points on the
surface
$$
\cal{V}:\;(p^2+q^2)^2=X^4+D
$$
is infinite.
}
\end{exam}

%\noindent
%More generally, consider the case of $(p_0,q_0)=(1,0)$ and $(c,d,e)=(D,0,1)$, that is, the surface
%\[ (p^2+b q^2)^2 = X^4 + D. \]
%Here, $\cal{C}:\;Y^2=X^4+D$ and $\cal{C}':\;w^2=2b(D-b^2 U^4)$. The above mapping satisfies
%$Y=X^2+b U^2$, so that $b \equiv Y-X^2 \bmod{{\Q^*}^2}$, and so $b \mid 2D$. If the rank of $\cal{C}$
%is $r$, then the set $\mathcal{B}$ of values of $Y-X^2 \bmod{{\Q^*}^2}$ has cardinality $2^r$;
%and so there are $2^r$ possibilities for $b \bmod{{\Q^*}^2}$ giving a point on $\cal{C}'$.
%For $b \in \mathcal{B}$ therefore, the surface
%\[ \cal{V}: \; (p^2+b q^2)^2 = X^4 + D \]
%has infinitely many points.

\section{Quartics of type $X^4+a_2 X^2 +a_0$} \label{sec3}
% Note: $y^2=x^4+a2*x^2+a0$ is isomorphic to the curve $Y^2=X(X^2 - 2*a2*X + a2^2-4*a0)$, which has $2$-isogenous curve $V^2=U(U^2 + a2*U + a0)$.
In this section we are interested in rational points $(X,Y)$ lying on the quartic curve
$$
\cal{C}:\;Y^2=X^4+a_{2}X^2+a_{0},
$$
with $Y$ representable by the quadratic form $p^2+bq^2$. Before stating any result, we observe that if the set $\cal{C}(\Q)$ is infinite then there are infinitely many rational points in $\cal{C}(\Q)$ with $Y$-coordinate represented by the quadratic form $p^2-(a_{2}^2-4a_0)q^2$. Indeed, this is a consequence of the addition law on $\cal{C}$ which says that if $P_{i}=[i]P_{1}=(X_{i},Y_{i})=(U_{i}/W_{i},V_{i}/W_{i}^2)$, where $P_{1}=(X_{1},Y_{1})$ is a point on $\cal{C}$ of infinite order, then
\begin{equation*}
\begin{array}{lll}
  U_{2i}=U_{i}^4-a_0 W_{i}^4, & & U_{2i+1}U_{1}=U_{i}^2U_{i+1}^2-a_0 W_{i}^2W_{i+1}^2,  \\
  V_{2i}=V_{i}^4-(a_{2}^2-4a_{0})U_{i}^2W_{i}^2, & & W_{2i+1}W_{1}=U_{i}^2W_{i+1}^2-U_{i+1}^2W_{i}^2,  \\
  W_{2i}=2U_iV_iW_i, & &V_{2i+1}V_{1}=V_{i}^2V_{i+1}^2-(a_{2}^2-4a_{0})U_{i}^2U_{i+1}^2W_{i}^2W_{i+1}^2.
\end{array}
\end{equation*}
%\begin{equation*}
%\begin{array}{lll}
%  U_{2i}=U_{i}^4-bW_{i}^4, & & U_{2i+1}U_{1}=U_{i}^2U_{i+1}^2-bW_{i}^2W_{i+1}^2,  \\
%  V_{2i}=V_{i}^4-(a_{2}^2-4a_{0})U_{i}^2W_{i}^2, & & W_{2i+1}W_{1}=U_{i}^2W_{i+1}^2-U_{i+1}^2W_{i}^2,  \\
%  W_{2i}=2U_iV_iW_i, & &V_{2i+1}V_{1}=V_{i}^2V_{i+1}^2-(a_{2}^2-4a_{0})U_{i}^2U_{i+1}^2W_{i}^2W_{i+1}^2.
%\end{array}
%\end{equation*}
for $i\in\N$; see \cite{Chud}. We thus see that the $Y$-coordinate of the point $P_{2i}$ is representable by the form $p^2-(a_{2}^2-4a_{0})q^2$. The same property holds for the point $P_{2i+1}$ provided that $Y_{1}$ is also so representable.
\begin{rem}{\rm
These formulas show the restrictive nature of the method presented in the previous section. Indeed, if we consider the surface $\cal{V}:\;(p^2+3q^2)^2=X^4+3$  then the method of the previous section fails, since the corresponding quartic $\cal{C}':\;w^2=-3(18U^4-6)$ is not locally solvable at 2 and thus has no rational points. However, for $f(X)=X^4+3$ we have $a_{0}=3, a_{2}=0, a_{2}^2-4a_{0}=-12$ and the above formulas show that the set of rational points on the surface $\cal{V}$ is infinite. (In order to prove this, it is enough to start with the point $P_{1}=(U_{1}/W_{1},V_{1}/W_{1}^2)=(1,2)$ and compute the even multiples.)}
\end{rem}

In light of this remark, and the above formulas, the question arises as to whether $Y$-coordinates of rational points on $C$ may be represented by other forms of type $p^2+bq^2$ with $b\neq -(a_{2}^2-4a_{0})$.
We suppose
\[ \mathcal{V}: \; (p^2+b q^2)^2 = H(X) = X^4 + a_2 X^2 + a_0, \quad a_0(a_2^2-4 a_0) \neq 0, \]
and assume the existence of a (finite) point $(p,q,X)=(p_0,q_0,r_0)$, $r_0\neq 0$, which implies
\[ a_0 = m^2 - r_0^4 - a_2 r_0^2, \qquad m=p_0^2+b q_0^2. \]
Set
\[ p=p_0+T, \; q=q_0 + u T, \; X= r_0 + v T. \]
There results a quartic polynomial in $T$ with constant term $0$. We make the coefficient of $T$ equal to $0$ by
taking
\[ v = \frac{2m(p_0 + b q_0 u)}{r_0 (a_2 + 2r_0^2)}. \]
We now have $C_2(u)T^2+C_3(u)T^3+C_4(u)T^4=0$. The discriminant $C_3(u)^2-4 C_2(u)C_4(u)$ must be square for rational $T$, so gives an equation
\begin{equation}
\label{sexticu}
 -8m r_0^2 (a_2 + 2r_0^2) F_2(u) F_4(u) = \square,
\end{equation}
where $F_2(u)$, $F_4(u)$ are polynomials of degree $2$, $4$ respectively.
% F_2(u)=(2*p0^4 + 2*b*p0^2*q0^2 - a2*r0^2 - 2*r0^4 + 4*b*p0^3*q0*u + 4*b^2*p0*q0^3*u + 2*b^2*p0^2*q0^2*u^2 + 2*b^3*q0^4*u^2 - a2*b*r0^2*u^2 - 2*b*r0^4*u^2)
If this sextic curve is to have genus at most $1$, then the discriminant with respect to $u$ of $F_2(u)F_4(u)$ must vanish. The factorization is
\begin{align*}
& 2^{56} a_0^2 b^{15} m^{34} r_0^{50} (a_2+2 r_0^2)^{31} (a_2^2-4 a_0)^4 (2 a_0+a_2 r_0^2) \times \\
& (16a_0^3+32a_0^2a_2r_0^2+16a_0^2r_0^4+24a_0a_2^2r_0^4-a_2^4r_0^4+32a_0a_2r_0^6+16a_0r_0^8),
\end{align*}
and so the three following distinct possibilities arise. \\ \\
%2^{56} a0^2*b^15*m^34*r0^50*(a2+2*r0^2)^51*(a2^2-4*a_0)^4*(2*a0+a2*r0^2)*(16*a0^3+32*a0^2*a2*r0^2+16*a0^2*r0^4+24*a0*a2^2*r0^4-a2^4*r0^4+32*a0*a2*r0^6+16*a0*r0^8)
\noindent
$\bullet$ If the factor $2 a_0+a_2 r_0^2$ vanishes, then
\[ (a_2,a_0) = \left( \frac{2(m^2-r_0^4)}{r_0^2}, -(m^2-r_0^4)\right), \]
% {a2,a0}={ 2(m^2-r0^4)/r0^2, -(m^2-r0^4) }
and the sextic (\ref{sexticu}) becomes
\[ -2b(-q_0 + p_0 u)^2 L_4(u) = \square, \]
% L4(u)= p0^8 + 4*b*p0^6*q0^2 + 6*b^2*p0^4*q0^4 + 4*b^3*p0^2*q0^6 + b^4*q0^8 - p0^4*r0^4 - 2*b*p0^2*q0^2*r0^4 - 4*b^2*p0*q0^3*r0^4*u + 2*b*p0^8*u^2 + 8*b^2*p0^6*q0^2*u^2 + 12*b^3*p0^4*q0^4*u^2 + 8*b^4*p0^2*q0^6*u^2 + 2*b^5*q0^8*u^2 - 2*b*p0^4*r0^4*u^2 + 2*b^2*p0^2*q0^2*r0^4*u^2 - 2*b^3*q0^4*r0^4*u^2 - 4*b^2*p0^3*q0*r0^4*u^3 + b^2*p0^8*u^4 + 4*b^3*p0^6*q0^2*u^4 + 6*b^4*p0^4*q0^4*u^4 + 4*b^5*p0^2*q0^6*u^4 + b^6*q0^8*u^4 - 2*b^3*p0^2*q0^2*r0^4*u^4 - b^4*q0^4*r0^4*u^4
with $L_4(u)$ quartic in $u$:
\begin{align*}
L_4(u)= & (m^4 - p_0^2(m + bq_0^2)r_0^4) - 4b^2p_0q_0^3r_0^4u + 2b(m^4 - (m^2 - 3 b p_0^2q_0^2)r_0^4)u^2+\\
        & -4b^2p_0^3q_0r_0^4 u^3 + b^2(m^4 - b(m + p_0^2)q_0^2r_0^4)u^4.
\end{align*}
%L_4(u)= (m^4 - p0^2*(m + b*q0^2)*r0^4) - 4*b^2*p0*q0^3*r0^4*u + 2*b*(m^4 - (m^2 - 3*b*p0^2*q0^2)*r0^4)*u^2 -4*b^2*p0^3*q0*r0^4*u^3 + b^2*(m^4 - b*(m + p0^2)*q0^2*r0^4)*u^4.
The discriminant of $L_4(u)$ is $-2^{8}b^{6}m^{14}r_0^8a_0^3$, so is non-zero.
It follows that there exists a non-constant rational map $\psi_{1}:\;\cal{C}_{1}\rightarrow \cal{V}_{1}$
where $\cal{C}_{1}$ is the genus $1$ curve given by
 \begin{equation}\label{curve1}
\cal{C}_{1}:\; -2b L_4(u) = \square
\end{equation}
and $\cal{V}_{1}$ corresponds to our choice of $a_{0}, a_{2}$.
\begin{exam}
{\rm
Take $b=3$, $(p_0,q_0,r_0)=(1,0,2)$. Then $(a_2,a_0)=(-\frac{15}{2}, 15)$ and we have a mapping $\psi_{1}: \cal{C}_{1} \rightarrow \cal{V}_{1}$ where
\[ \cal{C}_{1}: V^2=2(5+30 u^2-3 u^4), \qquad \cal{V}_1: (p^2+3 q^2)^2 = X^4 -\frac{15}{2} X^2 + 15 \]
given by
\begin{align*}
(p,&q, X) =\\
   &\left( \frac{ (5 + 3u^4 - uV) }{(-1 + u^2)(5 + 3u^2) }, \quad  -\frac{ u(-10 + 2u^2 + uV) }{(-1 + u^2)(5 + 3u^2) }, \quad
 \frac{ 2(5 + 3u^4 - uV)}{(-1 + u^2)(5 + 3u^2)} \right).
\end{align*}
% {p,q,X}={(5 + 3*u^4 - u*V)/((-1 + u)*(1 + u)*(5 + 3*u^2)), -((u*(-10 + 2*u^2 + u*V))/((-1 + u)*(1 + u)*(5 + 3*u^2))), (2*(5 + 3*u^4 - u*V))/((-1 + u)*(1 + u)*(5 + 3*u^2))}
Since $\cal{C}_{1}$ is elliptic of rational rank $2$, with generators $(-1,8)$, $(-3,8)$, it follows that there are infinitely many rational points on the elliptic curve
\[ Y^2 = X^4 - \frac{15}{2} X^2 + 15 \]
whose $Y$-coordinate can be represented by the quadratic form $p^2+3q^2$. We note that the curve $\cal{C}_{1}$ is 2-isogenous to the curve corresponding to $\cal{V}_{1}$.
}
\end{exam}

\noindent
$\bullet$ If the factor $a_2+2r_0^2$ vanishes, then
\[ (a_2,a_0)=(-2 r_0^2, m^2+r_0^4), \]
and we have to eliminate $u$ rather than $v$. There results a quartic equation of the form
\begin{equation}\label{curve2}
\cal{C}_{2}:\; 2b (-m^2+2b q_0^2 r_0^2 v^2 + b^2 q_0^4 v^4) = \square;
\end{equation}
the discriminant of the quartic is $-2^{14}b^{12}m^2q_0^{12}(m^2+r_0^4)^2$, and hence non-zero,
so that $\cal{C}_{2}$ is of genus $1$. Accordingly we have a non-constant rational map $\psi_{2}:\;\cal{C}_{2}\rightarrow \cal{V}_{2}$, where $\cal{V}_{2}$ corresponds to our choice of $a_{0}, a_{2}$.
\begin{exam}
{\rm
Take $b=2$, $(p_0,q_0,r_0)=(1,1,1)$. Then $(a_2,a_0)=(-2,10)$ and we have a mapping $\psi_{2}: \cal{C}_{2} \rightarrow \cal{V}_{2}$ where
\[ \cal{C}_{2}: U^2=-9+4v^2+4v^4, \qquad \cal{V}_{2}: (p^2+2 q^2)^2 = X^4 -2 X^2 + 10 \]
given by
\[ (p,q,X) =  \left( \frac{ 9 + 8v^3 - 4v^4 - 6U }{9 - 4v^4 }, \quad
 \frac{ 9 - 4v^3 - 4v^4 + 3U }{ 9 - 4v^4 },  \quad
 \frac{ 9 + 4v^4 -6Uv }{ 9 - 4v^4 } \right). \]
% {p,q,X}={(-9 + 6*U - 8*v^3 + 4*v^4)/((-3 + 2*v^2)*(3 + 2*v^2)), (-9 - 3*U + 4*v^3 + 4*v^4)/((-3 + 2*v^2)*(3 + 2*v^2)), -((9 - 6*U*v + 4*v^4)/((-3 + 2*v^2)*(3 + 2*v^2)))}
Since $\cal{C}_{2}$ is of rational rank $1$, with generator $(\frac{3}{2},\frac{9}{2})$, it follows that there are infinitely many rational points on the elliptic curve
\[ Y^2 = X^4 - 2X^2 + 10 \]
whose $Y$-coordinate can be represented by the quadratic form $p^2+2q^2$. Moreover, one can check that the curve $\cal{C}_{2}$ is 2-isogenous to the curve corresponding to $\cal{V}_{2}$.
}
\end{exam}

\noindent
$\bullet$ The last factor in the discriminant defines a quartic curve in $a_0,a_2$ of genus $0$ over $\Q(r_0)$.
A parametrization is given by
\[ (a_2, \; a_0) = \left(\frac{2t(r_0^2+r_0 t+t^2)}{r_0}, \; t^4\right). \]
The condition that $(p_0,q_0,r_0)$ be a solution becomes
\[ (p_0^2+b q_0^2)^2 = (r_0+t)^2 (r_0^2+t^2), \]
so that necessarily $r_0^2+t^2=\square$. Put $t=\frac{r_0 (s^2-1)}{2s}$, where $s \neq 0, \pm 1$. Then
\[ (a_2, \; a_0) = \left(\frac{r_0^2 (s^2-1)(s^4+2s^3+2s^2-2s+1)}{4s^3}, \;\; r_0^4 \frac{(s^2-1)^4}{16s^4}\right), \]
% (a2,a0)=(r0^2*(s^2-1)*(s^4+2s^3+2s^2-2s+1)/(4*s^3), r0^4*(s^2-1)^4/(16*s^4))
\[ p_0^2+b q_0^2 = \frac{(s^2+1)(s^2+2s-1)}{4s^2} r_0^2. \]
% p0^2+b*q0^2 = (s^2+1)*(s^2+2*s-1)/(4*s^2)*r0^2
The quartic $H(X)$ now takes the form
\[ \frac{1}{16s^4} (r_0^2 (s+1)^3(s-1) + 4s X^2) (r_0^2 (s+1)(s-1)^3 + 4s^3 X^2); \]
that is,
\[ \mathcal{V}:\;16s^4(p^2+b q^2)^2 = (r_0^2(s+1)^3(s-1)+4s X^2)(r_0^2(s+1)(s-1)^3+4s^3 X^2), \]
where we demand a solution $(p_0,q_0,r_0)$ of
\begin{equation}
\label{sconic}
p_0^2+b q_0^2 = \frac{(s^2+1)(s^2+2s-1)}{4s^2} r_0^2.
\end{equation}
% p0^2+b*q0^2 - (s^2+1)*(s^2+2*s-1)/(4*s^2)*r0^2
The sextic (\ref{sexticu}) has become
\begin{equation}
\label{sexticured} -b s (-q_0 + p_0 u)^2 g_1(u) g_2(u) = \square,
\end{equation}
% -b s g1(u) g2(u)
where
\begin{align*}
g_1(u) & = ((1-s^2)p_0^2 + 2s b q_0^2) + 2(1-2s-s^2) b p_0 q_0 u + b(2s p_0^2+(1-s^2)b q_0^2)u^2, \\
% g1= (p0^2 + 2*b*q0^2*s - p0^2*s^2) + 2b p0 q0 (1 - 2*s - s^2)u + b(b*q0^2 + 2p0^2*s - b*q0^2*s^2)*u^2
g_2(u) & = 2(1-s^2)^2 p_0^2+ (1+s^2)^2 b q_0^2 + 2(1-6s^2+s^4) b p_0 q_0 u + b((1+s^2)^2 p_0^2 + 2(1-s^2)^2 b q_0^2)u^2.
\end{align*}
% g2= 2*p0^2 + b*q0^2 - 4*p0^2*s^2 + 2*b*q0^2*s^2 + 2*p0^2*s^4 + b*q0^2*s^4 + 2*b*p0*q0*u - 12*b*p0*q0*s^2*u + 2*b*p0*q0*s^4*u + b*p0^2*u^2 + 2*b^2*q0^2*u^2 + 2*b*p0^2*s^2*u^2 - 4*b^2*q0^2*s^2*u^2 + b*p0^2*s^4*u^2 + 2*b^2*q0^2*s^4*u^2
% disc(g1) = 8*b*(p0^2+b*q0^2)^2*s*(s^2-1)
% disc(g2) = -8*b*(p0^2+b*q0^2)^2*(s^4-1)^2
% Resultant(g1,g2,u) = b^2*(p0^2+b*q0^2)^4*(s^2-1)^2*(-1+2*s+s^2)^4
The discriminants of $g_1,g_2$ cannot vanish, and $g_1,g_2$ have no common root.
Accordingly, there is a non-constant rational map
$\psi_{3}:\;\cal{C}_{3}\rightarrow \cal{V}_{3}$, where $\cal{C}_{3}$ is the genus $1$ curve
\begin{equation}\label{scurve}
\cal{C}_{3}:\; -b s g_1(u) g_2(u) = \square
\end{equation}
and $\cal{V}_{3}$ corresponds to our choice of $a_{0}, a_{2}$.
%In these three possibilities above, we would like $C$ as constructed in each case to have positive rank.
%It may or may not happen, as the following show.
\begin{exam}
{\rm
Take $b=-5$; $s=\frac{1}{2}$. The conic (\ref{sconic}) is $p_0^2-5 q_0^2 = \frac{5}{16} r_0^2$ with point $(p_0,q_0,r_0)=(\frac{5}{8},\frac{1}{8},1)$.
Then $(a_2,a_0)=(-\frac{39}{32}, \; \frac{81}{256})$, and
\[ \cal{V}_{3}: \; (p^2-5 q^2)^2 = X^4 - \frac{39}{32} X^2 + \frac{81}{256}. \]
The associated elliptic quartic at (\ref{scurve}) above is now
\[ \cal{C}_{3}:\;2(-11 - 10 u + 85 u^2) (-13 - 14 u + 107 u^2) = \square, \]
with point at $u=-\frac{1}{3}$, and cubic model
\begin{equation}
\label{ucurve}
y^2 = x^3 + 588x^2 + 36x
\end{equation}
of rank $1$, generator $(36,-900)$.
On replacing $X$ by $\frac{1}{8} X$, it follows that there are infinitely many
rational points on the elliptic curve
\[ Y^2 = X^4 - 78 X^2 + 1296 =  (X^2 - 24)(X^2 - 54) \]
whose $Y$-coordinate can be represented by the quadratic form $p^2-5 q^2$. As in the previous examples, one can check that the curve $\cal{C}_{3}$ is 2-isogenous to the curve corresponding to $\cal{V}_{3}$.
}
\end{exam}
%Remark: likely the elliptic quartic at (\ref{sexticured}) will always be $2$-isogenous to the elliptic quartic
%defined by
%\[ Y^2 = \frac{(s^2-1)^4}{(16s^4)} X^4 + \frac{2(s^2-1)(s^4+2s^3+2s^2-2s+1)r_0^2}{8s^3} X^2 + r_0^4. \]
%% Y^2=(s^2-1)^4/(16*s^4)*X^4 + 2*r0^2*(s^2-1)*(s^4+2*s^3+2*s^2-2*s+1)/(8*s^3)*X^2 + r0^4
%This latter curve maps to
%\[ y^2 = x^3 - \frac{2s(1-2s+2s^2+2s^3+s^4)}{(s^2-1)^3}x^2 + \frac{s^2(1+s^2)^2(-1+2s+s^2)^2}{(s^2-1)^6}x \]
%with $\Z_2 \times \Z_2$ torsion.
%However, the following example shows that the $X$-quartic can have points, while the $u$-quartic
%does not have points (but the Jacobian of the $u$-quartic {\it is} $2$-isogenous to the $X$-quartic). \\ \\
\noindent
\begin{exam}
{\rm
Take $b=-17$, $s=5$.  The conic (\ref{sconic}) is $p_0^2 - 17 q_0^2 = \frac{221}{25} r_0^2$,
with point $(p_0,q_0,r_0)=(-\frac{119}{40}, -\frac{1}{40},1)$. But the $u$-quartic at (\ref{scurve}) is
\[ y^2 = 102(10001 - 4046 u + 71009 u^2)(-239735 + 28322 u + 2388313 u^2) \]
which is locally unsolvable at $17$.  The $X$-quartic here is
\[ Y^2 = H(X) = X^4 + \frac{5496}{125} X^2 + \frac{20736}{625}, \]
which has cubic model $y^2 = x^3 + 12270x^2 + 31212000x$ of rank $1$.
So the $X$-quartic has rank 1, but the $u$-quartic has no points.
The Jacobian of the $u$-quartic {\it is} $2$-isogenous to the $X$-quartic.
}
\end{exam}
Using exactly the same type of reasoning as in the proof of Theorem \ref{firstthm} one can prove that for each $i\in\{1,2,3\}$ the image of the map $\psi_{i}:\;\cal{C}_{i}\rightarrow \cal{V}_{i}$ is not finite provided that $\cal{C}_{i}(\Q)$ is infinite. We omit these details and summarize our results in the following:

\begin{thm}\label{biquartic1}
Let $m=p_0^2+b q_0^2$.
Suppose the set of rational points on the curve
\begin{align*}
\cal{C}_1: \; & -2b ((m^4 - p_0^2(m + bq_0^2)r_0^4) - 4b^2p_0q_0^3r_0^4u + 2b(m^4 - (m^2 - 3 b p_0^2q_0^2)r_0^4)u^2\\
              & -4b^2p_0^3q_0r_0^4 u^3 + b^2(m^4 - b q_0^2(m + p_0^2)r_0^4)u^4) = \square,
\end{align*}
is infinite. Then the set of rational points on the surface
\[ V_1: \; (p^2+b q^2)^2 = X^4 + \left( \frac{2(m^2-r_0^4)}{r_0^2} \right) X^2 - (m^2-r_0^4) \]
is infinite too (with the convention at the head of Section~\ref{sec2}).
\end{thm}

\begin{thm}\label{biquartic2}
Let $m=p_0^2+b q_0^2$.
Suppose the set of rational points on the curve
\[ \cal{C}_{2}:\; 2b (-m^2+2b q_0^2 r_0^2 v^2 + b^2 q_0^4 v^4) = \square, \]
is infinite. Then the set of rational points on the surface
\[ V_2: \; (p^2+b q^2)^2 = X^4 -2 r_0^2  X^2 + (m^2+r_0^4) \]
is infinite too (with the convention at the head of Section~\ref{sec2}).
\end{thm}

\begin{thm}\label{biquartic3}
Let $p_0^2+b q_0^2 = \frac{(s^2+1)(s^2+2s-1)}{4s^2} r_0^2$.
Suppose the set of rational points on the curve $\cal{C}_{3}$ at (\ref{scurve}) is infinite. Then the set of rational points on the surface
\[ V_3: \; (p^2+b q^2)^2 = X^4 + \frac{r_0^2(s^2-1)(s^4+2s^3+2s^2-2s+1)}{4s^3}X^2 + r_0^4 \frac{(s^2-1)^4}{16s^4} \]
is infinite too (with the convention at the head of Section~\ref{sec2}).
\end{thm}

%\begin{thm}\label{biquartic}
%Let $i\in\{1,2,3\}$ and suppose the set of rational points on the curve $\cal{C}_{i}$ at {\rm (\ref{curve1}), (\ref{curve2}), (\ref{scurve})} is infinite. Then the set of rational points on the surface $\cal{V}_{i}$ is infinite too (with the convention at the head of
%%this section).
%Section~\ref{sec2}).
%\end{thm}

\section{A remark on inhomogeneous representation}\label{sec4}

The ideas of Section~\ref{sec2} and Section~\ref{sec3} can be used to investigate representing the $Y$-coordinate of points on the curve
$\cal{C}:\;Y^2=A_4X^4+...+A_0:=f(X)$ by the inhomogeneous quadratic form $ax^2+by^2+c$, $a,b,c\in\Z$, $ab \neq 0$.
Consider the quartic surface
$$
\cal{V}:\;(ap^2+bq^2+c)^2=f(X),
$$
where we assume $\op{Disc}_{X}(f) \neq 0$.
As before, on multiplying through by $a^2$, we may suppose without loss of generality that $a=1$; and again when we refer to $\cal{V}(\Q)$ being infinite, we mean that $\pi_X(\cal{V}(\Q))$ is infinite, where $\pi_X: \cal{V}(\Q) \rightarrow \Q$ is the projection onto the $X$-coordinate.

Suppose that $(p_{0}, q_{0}, X_{0})$ lies on $\cal{V}$ with $f(X_{0}) \neq 0$. Without loss
of generality we may assume that $X_{0}=0$ and $p_{0}q_{0}\neq 0$. In particular
$A_{0}=(p_{0}^2+bq_{0}^2+c)^2\neq 0$. To shorten notation, put $m=p_0^2+bq_0^2+c$ so that $\cal{V}$ takes the form
$$
\cal{V}:\;(p^2+bq^2+c)^2=A_{4}X^4+A_{3}X^3+A_{2}X^2+A_{1}X+m^2.
$$
To find further points on $\cal{V}$, put
\begin{equation}\label{inhsub}
X=T,\quad p=p_{0}+\frac{U+p_{0}A_{1}}{4(m-c)m}T,\quad q=q_{0}-\frac{p_{0}U-bq_{0}^2A_{1}}{4bq_{0}(m-c)m}T,
\end{equation}
where $U, T$ are to be determined. This substitution gives
$$
(p^2+bq^2+c)^2-f(X)=\frac{T^2}{256b^2q_{0}^4(m-c)^2m^4}F(T),
$$
where $F(T)=B_{0}+B_{1}T+B_{2}T^2$ and $B_{i}\in\Z[U]$ with $\op{deg}_{U}B_{0}=\op{deg}_{U}B_{1}=2$
and $\op{deg}_{U}B_{2}=4$. Thus the point $(p,q,X)$ with coordinates given by (\ref{inhsub}) lies
on $\cal{V}$ if and only the equation $F(T)=0$ has rational roots, which is equivalent to the
discriminant $\op{Disc}_T(F)$ being a square. This corresponds to considering the curve in the $(U,V)$
plane given by the equation
$$
\cal{C}:\;V^2=B_{1}^2(U)-4B_{0}(U)B_{2}(U)=-128bq_{0}^2(m-c)m^3G(U).
$$
Here $G$ is monic of degree 6 (and does not contain any odd power of $U$).
In consequence, if $G$ has no multiple roots then $\cal{G}$ has only finitely
many rational points. We thus seek $\op{Disc}_{U}(G)=0$. We have that
\begin{equation*}
\op{Disc}_{U}(G)=-2^{42} b^{15} q_0^{30} (m-c)^{12} m^{12} (\op{Disc}_{X}(f))^2 H,
\end{equation*}
where
\begin{align*}
H=&\;-8(m-c)m(A_{1}^4-256A_{4}(m-c)^2m^4)A_{2}+A_{1}^6-64(m-c)^2m^2A_{1}^3A_{3}+\\
  &\;256(2c-3m)(m-c)^2m^3A_{4}A_{1}^2-512(m-c)^3m^5A_{3}^2,
\end{align*}
% -8(m-c)m(A1^4-256 A4(m-c)^2 m^4)A2+A1^6-64(m-c)^2 m^2 A1^3 A3+256(2c-3m)(m-c)^2 m^3 A4 A1^2-512(m-c)^3 m^5 A3^2
and $\op{Disc}_{U}(G) = 0$ if and only if $H=0$.\\ \\
We consider two cases: (I) $A_{1}^4-256 A_{4}(m-c)^2 m^4=0$; and (II) $A_{1}^4-256A_{4}(m-c)^2 m^4\neq 0$. \\ \\
{\bf Case I:} $A_{1}^4-256A_{4}(m-c)^2 m^4=0$.\\
Then $A_{4}=\frac{A_{1}^4}{256(m-c)^2m^4}$ and $mH=-2(m-c)(A_{1}^3 - 16m^3(m-c)A_{3})^2$, whence
$A_{3}=\frac{A_{1}^{3}}{16(m-c)m^3}$.
%$$
%A_{3}=\frac{A_{1}^{3}}{16(m-c)m^3}.
%$$
Since we are interested only in polynomials $f$ with $\op{deg}_{X}f\in\{3,4\}$, then necessarily $A_{1}\neq 0$.
Summing up: after some simplification, if
$$
f_{1}(X)=\frac{A_{1}^4}{256(m-c)^2m^4}X^4 + \frac{A_{1}^{3}}{16(m-c)m^3}X^3+A_{2}X^2+A_{1}X+m^2
$$
% f1=A1^4/(256(m-c)^2m^4)X^4 + A1^3/(16(m-c)m^3)X^3+A2 X^2+A1 X+m^2
then there exists a non-constant rational map $\phi_{1}:\;\cal{C}_{1}\rightarrow \cal{V}_{1}$, where (on replacing $U$ by $q_0 U$)
\begin{align}\label{curvec1}
\cal{C}_{1}: \; V^2=-&2b(m-c)m \; \times\\
                   &\left( m U^4+b m (3A_{1}^2 - 8m A_{2}(m-c))U^2 -2b^2 A_1^2((2c-3m)A_{1}^2+8(m-c)m^2A_{2}) \right)\notag
\end{align}
% C1: V^2=-2b(m-c)m ( m U^4+b m (3A1^2 - 8m A2(m-c))U^2 -2b^2 A1^2((2c-3m)A1^2+8(m-c)m^2A2) )
and $\cal{V}_{1}:\;(p^2+bq^2+c)^2=f_{1}(X)$. \\ \\
{\bf Case II}: $A_{1}^4-256A_{4}(m-c)^2m^4\neq 0$. \\
We solve the equation $H=0$ with respect to $A_{2}$ and get
\begin{equation}\label{A2}
A_{2}=\frac{A_{1}^6+64(m-c)^2m^2A_{1}^3A_{3}+256(2c-3m)(m-c)^2m^3A_{1}^2A_{4}-512(m-c)^3m^5A_{3}^2}{8m(m-c)(A_{1}^4-256A_{4}(m-c)^2m^4)}.
\end{equation}
% 8m(m-c) A2=(A1^6+64(m-c)^2 m^2 A3 A1^3+256(2c-3m)(m-c)^2 m^3 A4 A1^2-512(m-c)^3 m^5 A3^2)/ (A1^4-256A4(m-c)^2 m^4)
Summing up: after some simplification, if
$$
f_{2}(X)=A_{4}X^4+A_{3}X^3+A_{2}X^2+A_{1}X+m^2,
$$
where $A_{2}$ is given by (\ref{A2}), then there exists a non-constant rational map $\phi_{2}:\;\cal{C}_{2}\rightarrow \cal{V}_{2}$, where (on replacing $U$ by $q_0 U$),
\begin{equation}
\label{curvec2}
\cal{C}_{2}:\;V^2=2bm(m-c)(256A_{4}(m-c)^2m^4-A_{1}^4)(C_{4}U^4+C_{2}U^2+C_{0}),
\end{equation}
with
\begin{align*}
C_{4}=& A_{1}^4-256(m-c)^2m^4A_{4},\\
C_{2}=&2b(A_{1}^6-32(m-c)^2m^2(A_{3}A_{1}^3+ 8cmA_{4}A_{1}^2-8(m-c)m^3A_{3}^2)),\\
%C_{0}=&b^2[A_{1}^8-64(m-c)^2m^2(A_{3}A_{1}^5+8(2 c - m)mA_{4}A_{1}^4+16(c - m) m^3A_{3}A_{1}^2\\
%      &\quad\quad+256(c - m)^2 m^4A_{3}A_{4}A_{1}-1024(c - m)^2m^6A_{4}^2)]. \\
C_0=& b^2(A_1^8 -64(m-c)^2 m^2 A_1^5A_3 + 512(m-c)^2 m^3 (m-2c) A_1^4 A_4+\\
    & +1024(m-c)^3 m^5 A_1^2 A_3^2 -16384(m-c)^4 m^6 A_4 (A_1 A_3-4 A_4 m^2) )
\end{align*}
% C4=A1^4-256*(m-c)^2*m^4*A4
% C2=2*b*q0^2*(A1^6-32*(m-c)^2*m^2*(A3*A1^3+8*c*m*A4*A1^2-8*(m-c)*m^3*A3^2))
% C0=b^2*q0^4*(A1^8 -64*(m-c)^2*m^2*A1^5*A3+512*(m-c)^2*m^3*(m-2*c)*A1^4*A4+1024*(m-c)^3*m^5*A1^2*A3^2-16384*(m-c)^4*m^6*A4*(A1*A3-4*A4*m^2))
and $\cal{V}_{2}:\;(p^2+bq^2+c)^2=f_{2}(X)$.\\ \\
We can summarize our result in the following:

\begin{thm}
Let $m=p_0^2+b q_0^2+c$.  Suppose the set of rational points on the curve $\cal{C}_{1}$ at (\ref{curvec1})
is infinite. Then the set of rational points on the surface
\[ \cal{V}_1: \; (p^2+b q^2 +c)^2 = \frac{A_1^4}{256(m-c)^2m^4} X^4 +\frac{A_1^3}{16(m-c)m^3} X^3 + A_2 X^2 + A_1 X + m^2 \]
is infinite too (with the convention at the head of this section).
\end{thm}

\begin{thm}
Let $m=p_0^2+b q_0^2+c$.  Suppose the set of rational points on the curve $\cal{C}_{2}$ at (\ref{curvec2}) is infinite. Then the set of rational points on the surface
\[ \cal{V}_{2}: \; (p^2+b q^2+c)^2 = A_4 X^4+A_3 X^3+A_2 X^2+A_1 X+m^2,  \]
where $A_2$ is given by (\ref{A2}),
is infinite too (with the convention at the head of this section).
\end{thm}

%\begin{thm}
%Let $i\in\{1,2\}$ and suppose the set of rational points on the curve $\cal{C}_{i}$ is infinite. Then the set of rational points on the surface $\cal{V}_{i}$ is infinite too (with the convention at the head of this section).
%\end{thm}

\begin{exam}
{\rm
Take $(b,c,p_0,q_0)=(1,1,1,1)$, $(A_1,A_2)=(2,2)$. Then $\phi_1: \cal{C}_1 \rightarrow \cal{V}_1$ where
\[ \cal{C}_{1}: V^2=130+63 U^2-3 U^4, \quad \cal{V}_{1}: (p^2+q^2+1)^2 = 9+24X+288X^2+16X^3+4X^4 \]
given by
\begin{align*}
(p,& q, X) = \\
   &\left( \frac{ -2U^2+U^3 -(1+U)V }{ U(2 + U^2) }, \quad \frac{ 2U^2+U^3 -(1-U)V }{U(2 + U^2) }, \quad \frac{ -(2U+V) }{ U(2 + U^2) } \right).
\end{align*}
% {p,q,X}={(-V - V*U - 2*U^2 + U^3)/(U*(2 + U^2)), (-V + V*U + 2*U^2 + U^3)/(U*(2 + U^2)), -(V + 2*U)/(U*(2 + U^2))}
Now $\cal{C}_{1}$ has rational rank $2$, with independent points
$$
\left( \frac{9}{2},\frac{53}{4} \right),\;\left( \frac{152129}{152882}, \frac{321697804123}{152882^2} \right),
$$ and it follows that the curve $Y^2= 9+24X+288X^2+16X^3+4X^4$ has infinitely many points with $Y$-coordinate represented by the form $p^2+q^2+1$. We note that the curve $\cal{C}_{1}$ is 2-isogenous with the curve related to $\cal{V}_{1}$.
}
\end{exam}

\begin{rem}\label{genmethod}
{\rm
Essentially the same method as presented in the last three sections can be applied to curves defined by more general equations of the form
$$
\cal{C}:\;Y^2+f_{2}(X)Y=f_{4}(X),
$$
where $f_{i}\in\Z[X]$ and $\op{deg}f_{i}=i$ for $i=2,4$. However, in order to avoid long computations we have presented our approach only for those $\cal{C}$ with $f_{2}(X)\equiv 0$.
}
\end{rem}

\section{The Diophantine equation $Y^2=f(ap^2+bq^2)$ with $f$ cubic}\label{sec5}
Motivated by the results above, we investigate whether any variation
of the method can be used in other situations. Specifically, we consider
the question of when the $X$-coordinates of points on an elliptic curve of type $Y^2=f(X)$, $f$ cubic,
can be represented by the quadratic form $a x_1^2+b x_2^2$.
First, observe that the curve
$$
E_{1}:\;y^2=x^3+Ax^2+Bx,
$$
with $A, B\in\Q$, $A^2-4B\neq 0$, is 2-isogenous to the curve
$$
Y^2 = X^3-2A X^2+(A^2-4B)X
$$
under $(X,Y)=(\frac{y^2}{x^2}, \; \frac{y(B-x^2)}{x^2})$.
Accordingly, $E_1$ is 2-isogenous to the curve
$$
E_2:\;Y^2=\left(\frac{X-b}{a}\right)^3-2A\left(\frac{X-b}{a}\right)^2+(A^2-4B)\left(\frac{X-b}{a}\right)
$$
with isogeny given by
$$
\phi:\;E_{1}\ni (x,y)\mapsto \left(a\left(\frac{y}{x}\right)^2+b,\quad \frac{y(B-x^2)}{x^2}\right)\in E_{2}.
$$
Thus if $E_1$ has positive rank, then $E_2$ has infinitely many rational points $(X,Y)$ with $X$ represented by the form $ap^2+bq^2$.
\\ \\
Consider the surface
\begin{equation*}
\cal{W}:\;Y^2=f(ap^2+bq^2),
\end{equation*}
where $f(X)=c_{3}X^3+c_{2}X^2+c_{1}X+c_{0} \in \Z[X]$ is non-singular. Without loss of generality, on multiplying by $c_3^2$, we may take $c_{3}=1$.
As in the case of the surface $\cal{V}$ considered in Section~\ref{sec2}, when we say that the set
$\cal{W}(\Q)$ is infinite, we mean that the set $\pi_{Y}(\cal{W}(\Q))$ is infinite, where the map
$\pi_{Y}:\;\cal{W}(\Q)\ni (p,q,Y)\mapsto Y\in\Q$ is simply the projection onto the $Y$-line. \\ \\
%Indeed, our assumption is reasonable because if $(p_{0},q_{0},Y_{0})\in\cal{W}(\Q)$ then each rational point on the (rational parameterizable!) curve $ap_{0}^2+bq_{0}^2=ap^2+bq^2$ corresponds with $Y_{0}$.
We assume that $(p_0, q_{0}, Y_{0})$ is a rational point on $\cal{W}$ with $Y_{0}\neq 0$.
Now
$$
f(0)=c_{0}=Y_0^2-(m^3+c_{2}m^2+c_{1}m), \qquad m=a p_0^2+b q_0^2.
$$
Computing the discriminant of $f$ with this value of $c_{0}$,
\begin{align}\label{discf}
\op{Disc}_{X}(f)= &\; -27Y_0^4-2(c_2 + 3m)(2c_2^2-9c_1-6c_2m-9m^2)Y_0^2 \\
                  &\; -(c_1+2c_2m+3m^2)^2(-c_2^2+4c_1+2c_2m+3m^2) \nonumber
\end{align}
%discf= -27*c3^2*Y0^4-2*(c2 + 3*c3*m)*(2*c2^2 - 9*c1*c3 - 6*c2*c3*m - 9*c3^2*m^2)*Y0^2-((c1 + 2*c2*m + 3*c3*m^2)^2*(-c2^2 + 4*c1*c3 + 2*c2*c3*m + 3*c3^2*m^2)).

We will present a method which sometimes allows us to prove that the set $\cal{W}(\Q)$ is infinite.
In order to find more points on $\cal{W}$ set
% we are looking for the parameters $U, V, r, s, T$ such that the point $(p, q, Y)$, with
\begin{equation}
\label{pqY}
p=T+p_{0},\quad q=UT+q_{0}, \quad Y=VT^3+rT^2+sT+Y_{0},
\end{equation}
and demand that $(p,q,Y)$ lie on the surface $\cal{W}$. This gives
\begin{equation*}
Y^2-f(ap^2+bq^2)=\sum_{i=1}^{6}C_{i}T^{i},
\end{equation*}
where $C_{i}$ for $i=1,\ldots, 6$, is a polynomial in $\Q[U,V,r,s]$ with coefficients depending on
$p_{0}, q_{0}, Y_{0}$. The system $C_{1}=C_{2}=C_{3}=0$ is linear in $V, r, s$ and has exactly one
solution, of the form
\begin{equation}\label{qrs:substitution}
V=\frac{Q(U)}{2Y_{0}^5},\quad r=\frac{R(U)}{2Y_{0}^{3}},\quad s=\frac{S(U)}{Y_{0}},
\end{equation}
where $Q, R, S$ are polynomials in $U$ with coefficients depending on $p_{0}, q_{0}, Y_{0}$.
Moreover, $\op{deg}P=3,\;\op{deg}Q=2$ and $\op{deg}S=1$.
%The explicit forms of $P, Q, R$ are not presented because of their size.
Substituting from (\ref{pqY}) into the polynomial $\sum_{i=1}^{6}C_{i}T^{i}$, using $V,r,s$ from
(\ref{qrs:substitution}), there results
\begin{equation}\label{aftersubstitution}
H_{U}(T):=\frac{A_{4}}{4Y_{0}^{6}}+\frac{A_{5}}{2Y_{0}^{8}}T+\frac{A_{6}}{4Y_{0}^{10}}T^2=0,
\end{equation}
%{\color{red}
%where $A_{i}(U)=2^{\alpha_{i}}Y_{0}^{2i-2}C_{i}(U,V,r,s)$ for $i=4, 5, 6$, and $V, r, s$ are given
%in (\ref{qrs:substitution}). Moreover, $\alpha_{4}=\alpha_{6}=2, \alpha_{5}=1$.
%}
It is clear that a necessary condition for solvability of (\ref{aftersubstitution}) for some fixed $U\in\Q$ is that the discriminant, say $\Delta_{U}(H)$, of the polynomial $H_{U}$ should be a square in $\Q$. We have that $\Delta_{U}(H)$ is a polynomial in $U$ of degree 10. More precisely,
\begin{equation*}
(2Y_{0}^{8})^{2}\Delta_{U}(H)=A_{5}^{2}-A_{4}A_{6}=4H_{1}(U)H_{2}(U),
\end{equation*}
where $\op{deg}H_{1}=4$ and $\op{deg}H_{2}=6$.
%This factorization was computed with the help of the {\sc Magma} computational package.
We consider the curve
\begin{equation}\label{Discriminantcurve}
\cal{C}:\;W^2=H_{1}(U)H_{2}(U).
\end{equation}

\noindent
For general choice of $c_{1}, c_{2}, Y_{0}, p_{0}$ and $q_{0}$ the curve $\cal{C}$ is a hyperelliptic curve of genus 4. Due to Faltings Theorem we know that a necessary condition for $\cal{C}$ to have infinitely many rational points is that $\op{genus}(\cal{C})\leq 1$. This immediately implies that the discriminant of the polynomial $H_{1}(U)H_{2}(U)$ with respect to $U$ must be 0. In fact we need to have an equality of the following form: $H_{1}(U)H_{2}(U)=g(U)h(U)^2$, where the degree of $g$ is at most 4.

Since $\op{Disc}_{U}(H_{1}H_{2})=\op{Disc}_{U}(H_{1})\op{Disc}_{U}(H_{2})\op{Res}_{U}(H_{1},H_{2})^{2}$, we compute independently the discriminants of $H_{1}$, $H_{2}$, and the resultant of these two polynomials. To shorten the notation, introduce new quantities:
$$
y_{0}=Y_{0}^2,\quad Z=3m^2+2c_{2}m+c_{1}.
$$
A quick computation reveals that
$$
\op{Disc}_{U}(H_{1})=2^{12}(ab)^{6}m^{4}y_{0}^{10}h_{11}h_{12}^{2}h_{13},
$$
where
\begin{align*}
h_{11}=&\;-4(c_{2} + 3m)y_{0}+Z^2,\\
h_{12}=&\;9y_0^2-4(3m+c_{2})y_0Z+Z^{3},\\
h_{13}=&\;-4(6m+c_2)y_0^3+8m(4m+c_2)y_0^2 Z+\\
       &-4m^2(3m+c_2)y_0 Z^2+Z^2(y_0-m Z)^2.
\end{align*}

Computation of the expression for $\op{Disc}_{U}(H_{2})$ took much longer; indeed, about 64 hours of CPU time. We have
$$
\op{Disc}_{U}(H_{2})=2^6(ab)^{15}m^{12}y_0^{21}(\op{Disc}_X(f))^2 h_{21}^{8} h_{23},
$$
where
\begin{align*}
h_{21}=&\;8y_0^2-4(3m+c_{2})y_0Z+Z^3,\\
h_{23}=&\; -y_0^5 + 2m(2c_2^2 + 28c_2m + 98m^2 - Z)y_0^4 +\\
       &-m(16c_2^2m + 160c_2m^2 + 336m^3 + c_2Z + 11m Z)y_0^3 Z+ \\
       & + 8m^2(2c_2^2m + 12c_2m^2 + 18m^3 + c_2Z + 5m Z)y_0^2 Z^2+\\
       & -m^2(8c_2m + 24m^2 + Z)y_0 Z^4 + m^3Z^6 .
\end{align*}
Finally, the expression for the resultant of the polynomials $H_{1}, H_{2}$ with respect to $U$ is the following:
\begin{equation*}
\op{Res}_{U}(H_{1},H_{2})=(ab)^{12}m^{12}y_0^{16}g_{1}^{8}g_{2}^{2},
\end{equation*}
where
\begin{align*}
g_{1}=&\;8y_0^2-4(3m+c_{2})Zy_0+Z^3,\\
g_{2}=&\; 27y_0^2 + 2(c_2 + 3m)(2c_2^2 + 12c_2m + 18m^2 - 9Z)y_0+\\
      &- (c_2^2 + 6c_2m + 9m^2 - 4Z)Z^2
\end{align*}
Observe that $h_{21}=g_{1}$. In consequence, if $h_{21}=0$ then $H_{1}$ and $H_{2}$ have a common factor and $H_{2}$ has at least one double root. \\ \\
We are interested only in $c_1, c_2, p_{0}, q_{0}, Y_{0}$, such that $(ap_{0}^2+bq_{0}^2)Y_{0}\neq 0$ and $\op{Disc}_{U}(H_{1}H_{2})=0$.  Observe that if $\op{Disc}_{U}(H_{1})=0$ and $\op{Disc}_{U}(H_{2})\neq 0$ then a necessary condition for $\cal{C}$ to have genus 1 is $\op{Res}_{U}(H_{1},H_{2})=0$. Similarly, if $\op{Disc}_{U}(H_{1})\neq 0$ and $\op{Disc}_{U}(H_{2})=0$ then a necessary condition for the curve $\cal{C}$ to have genus 1 is that either $H_{2}$ is the square of a cubic polynomial, or $\op{Res}_{U}(H_{1},H_{2})=0$. Finally, if $\op{Disc}_{U}(H_{1})\op{Disc}(H_{2})\neq 0$ then the only possibility for $\cal{C}$ to have genus 1 is $\op{Res}_{U}(H_{1},H_{2})=0$. \\ \\
We perform a case by case analysis on the equations $h_{1j}=0$ with $j=1,2,3$, $h_{2j}=0$ with $j=1,3$, and $g_2=0$. \\ \\
{\bf Case I}: suppose $h_{11}$ vanishes. Then since $y_0 \neq 0$,
\[ c_{2}=\frac{Z^2-12my_0}{4y_0}, \qquad c_1=Z-2 c_2 m-3 m^2 = \frac{6m^2 y_0+2y_0 Z-m Z^2}{2y_0}. \]
The corresponding discriminant $\op{Disc}_{X}(f)=\frac{1}{2}(Z^3-54 y_0^2)$,
which must be non-zero.
Substituting the expressions for $c_1,c_2$ into $H_1(U)$, we find $H_1(U)$ has the factor
$(a p_0+b q_0 U)^2$, and the cofactor, say $G_{1}$ (of degree 2) is irreducible.
Now
\begin{align*}
\op{Res}_U(H_1,H_2)= &\; 2^{22} (ab)^{12} m^{12} y_0^{32} (Z^3-54 y_0^2)^2 \neq 0, \\
\op{Disc}_U(H_2) = &\; 2^{28}(ab)^{15}m^{12}y_0^{41}(64m^3-y_0-2mZ)(Z^3-54y_0^2)^2,
\end{align*}
so that necessarily $\op{Disc}_U(H_2)=0$ implying $y_0=2m(32m^2-Z)$ and $\op{Disc}_{X}(f)=-\frac{1}{2}(24m^2-Z)(96m^2-Z)^2$. Accordingly, $Z\neq 24m^2, 32m^2, 96m^2$. There is now the factor
$(q_0-p_0 U)^2$ of $H_2(U)$ and irreducible cofactor $G_2(U)$ of degree $4$.
Since $\op{Res}_U(H_1,H_2) \neq 0$, it follows that $G_1,G_2$ cannot have a common root.
So $\cal{C}$ can have genus $1$ only if either $G_1$ or $G_2$ has a multiple root.
But
\begin{align*}
\op{Disc}_U(G_1) = &\; -2^{14} \cdot 3 \cdot (ab) m^6(24m^2-Z)(32m^2-Z)^5, \\
\op{Disc}_U(G_2) = &\; 2^{48} (ab)^{12} m^{28}(24m^2-Z)^2(32m^2-Z)^{25}(96m^2-Z),
\end{align*}
and by the above, neither can be zero.
In summary, the vanishing of $h_{11}$ does not lead to any curve of genus $1$.\\ \\
{\bf Case II}: suppose $h_{12}=0$. Let $\cal{C}_{12}$ be the curve
% h12=9*y0^2-4*(3*m+c2)*y0*(c1+2*c2*m+3*m^2)+(c1+2*c2*m+3*m^2)^3
$$
\cal{C}_{12}:\;h_{12}(c_{1},c_{2})=0
$$
defined over the function field $\Q(m,y_0)$ in the plane $(c_{1},c_{2})$. The genus of the curve $\cal{C}_{12}$ is 0 and because it contains the $\Q(m, y_0)$-rational point at infinity $(-2m:1:0)$ we can find a rational parametrization, for example:
%\begin{align}\label{h12zero}
%c_{1}=&\; \frac{-9my_0^2+2(6m^2+u)(3m^2+u)y_0-m(3m^2+u)^3}{2(3m^2+u)y_0}, \\
%c_{2}=&\; \frac{9y_0^2-12m(3m^2+u)y_0+(3m^2+u)^3}{4(3m^2+u)y_0}.
%\end{align}
%% (c1,c2)=( (-9*m*y0^2+2*(6*m^2+u)*(3*m^2+u)*y0-m*(3*m^2+u)^3)/(2*(3*m^2+u)*y0),(9*y0^2-12*m*(3*m^2+u)*y0+(3*m^2+u)^3)/(4*(3*m^2+u)*y0) )
%{\color{red}
\[ c_{1}= -\frac{9m y_0^2-2u(3m^2+u)y_0+m u^3}{2u y_0}, \qquad c_{2}= \frac{9y_0^2-12m u y_0+u^3}{4u y_0}. \]
% (c1,c2) = ( -(9*m*y0^2-2*u*(3*m^2+u)*y0+m*u^3)/(2*u*y0), (9*y0^2-12*m*u*y0+u^3)/(4*u*y0) )
With these values of $c_1$ and $c_2$, $H_1(U)$ is divisible by the square of a quadratic irreducible over $\Q[u,m,y_0]$. Thus for $\cal{C}$ to have genus $1$, we require
$\op{Disc}_{U}(H_{2})\op{Res}_{U}(H_{1},H_{2})=0$.
This implies the vanishing of at least one of $h_{21}(=g_1)$, $h_{23}$, or $g_{2}$.
But with $c_1,c_2$ as above, we have $h_{21}=-y_0^2 \neq 0$; and $g_2=-\op{Disc}_X(f) \neq 0$.
We are thus left with $h_{23}=0$, which has become
$$
h_{23}=\frac{y_0^4(81my_0^2-4u(9m^2+u)y_0+mu^2(4m^2+u))}{4u^2}.
$$
% (y0^4*(81*m*y0^2-4*u*(9*m^2+u)*y0+m*u^2*(4*m^2+u))/(4u^2)
As a polynomial in $y_0$, there is a non-zero rational root for $y_0$ precisely when the discriminant of the quadratic factor, namely $4u^3(4u-9m^2)$, is a perfect square.
Thus $u(4u-9m^2)=w^2$ for some $w$. This equation defines a conic
% u*(4*u-9*m^2)-w^2
over $\Q(m)$ in the $(u,w)$ plane, with parametrization
$$
u=\frac{9m^{2}}{4-v^2}, \quad w=\frac{9m^2v}{4-v^2},
$$
% (u,w)=(9*m^2/(4-v^2), 9*m^2*v/(4-v^2) )
where $v$ is a rational parameter. The roots for $y_0$ are now
\[ y_0 = \frac{m^3 (5-2v)}{(2-v)^2(2+v)}, \]
% y0 = m^3*(5-2v)/((2-v)^2*(2+v)),
and that obtained by replacing $v$ by $-v$.
The corresponding $H_1(U),H_2(U)$ take the form
\[ H_1(U)=(4-v^2)\mbox{(square)} l_2(U)^2, \quad H_2(U)=-ab (2+v) \mbox{(square)} l_1(U)^2 l_4(U), \]
for polynomials $l_i(U)$ of degree $i$ which are irreducible over $\Q(v)[m]$.
On setting $v=2+t$, we have
\begin{align*}
(c_2,&c_1,c_0)=\\
     &\left(-\frac{m(1+t)(-1+13t+5t^2)}{(4+t)t(-1+2t)}, \frac{m^2(1+t)^2(7+4t)}{(4+t)t(-1+2t)}, -\frac{m^3(1+t)^4}{(4+t)t^2(-1+2t)}\right).
\end{align*}
Summing up, if $f(X)=X^3+c_{2}X^2+c_{1}X+c_{0}$, then on setting $(A,B)=(a p_0^2,b q_0^2)$ (with $m=A+B$), and replacing $U$ by $\frac{q_0}{p_0} U$, there exists a non-constant rational map $\psi_1: \cal{C}_1 \rightarrow \cal{W}_1$, where
\begin{equation}
\label{psi1}
\cal{C}_1: W^2 = A B t (k_4 U^4+k_3 U^3+k_2 U^2+k_1 U+k_0), \qquad \cal{W}_1: Y^2=f(ap^2+bq^2),
\end{equation}
and
\begin{align*}
k_0 &= -A^2((A+B)^2 + 2(A^2-B^2)t + A(A-3B)t^2 - A B t^3), \\
k_1 &= -2A^2B t(4+t)(A+B+(A-B)t),  \\
k_2 &= A B( -2(A+B)^2 + 3(A^2-6A B+B^2)t^2 + (A^2-4A B+B^2)t^3), \\
k_3 &= 2A B^2 t(4+t)(-(A+B) + (A-B)t),  \\
k_4 &= -B^2((A+B)^2 - 2(A^2-B^2)t - B(3A-B)t^2 - A B t^3).
\end{align*}
% {-A^2((A+B)^2 + 2(A^2-B^2)t + A(A-3B)t^2 - A B t^3), -2A^2B t(4+t)(A+B+(A-B)t),  A B( -2(A+B)^2 + 3(A^2-6A B+B^2)t^2 + (A^2-4A B+B^2)t^3),2A B^2 t(4+t)(-(A+B) + (A-B)t), -B^2((A+B)^2 - 2(A^2-B^2)t - B(3A-B)t^2 - A B t^3)}
The discriminant of $\cal{C}_1$ is $16(ABmt)^{12}(1+t)^6(4+t)^2(1-2t)$, non-zero since $t \neq 0,-1,-4,\frac{1}{2}$; and thus $\cal{C}_1$ is of genus $1$. \\ \\
\noindent
{\bf Case III}: suppose $h_{13}=0$. Again, this defines a genus zero curve over the
% h13=-4*(6*c3*m+c2)*y0^3+(35*c3*m^2+10*c2*m+c1)*y0^2*(c1+2*c2*m+3*m^2)-2*m*(9*c3*m^2+4*c2*m+c1)*y0*(c1+2*c2*m+3*m^2)^2+(c1+2*c2*m+3*m^2)^4*m^2
field $\Q(m, y_0)$ in the plane $(c_{1}, c_{2})$. There is a triple singular point at infinity
$P_{0}=(-2m:1:0)$ which determines the following parametrization:
%\begin{align}\label{h13zero}
%c_{1}=&\frac{2(9m^2+u)y_0^3-5mu(4m^2+u)y_0^2+2m^2u^2(3m^2+2u)y_0-m^3u^4}{2(mu-y_0)^2y_0},\\
%c_{2}=&\frac{1}{2m}(u-3m^2-c_{1}),
%\end{align}
\begin{align}\label{h13zero}
c_{1}&=\frac{2(9m^2+u)y_0^3-5mu(4m^2+u)y_0^2+2m^2u^2(3m^2+2u)y_0-m^3u^4}{2(mu-y_0)^2y_0},  \\
c_{2}&=\frac{1}{2m}(u-3m^2-c_{1}),\notag
\end{align}
% (c1,c2)=( (2*(9*m^2+u)*y0^3-5*m*u*(4*m^2+u)*y0^2+2*m^2*u^2*(3*m^2+2*u)*y0-m^3*u^4)/(2*(m*u-y0)^2*y0), (u-3*m^2-c1)/(2*m) )
where $u$ is a rational parameter. With $c_{1}, c_{2}$ given by (\ref{h13zero}), then
$\op{Disc}_{X}(f)=-g_{2}$ and thus $g_{2}$ cannot be zero. Further, $H_{1}$ has
the factor $(q_0-p_0 U)^2$ and the cofactor, say $G_{1}$, is
of degree 2 and irreducible over $\Q[m, y_0]$. We can assume that $\op{Disc}_{U}(G_{1})\neq 0$,
because $\op{Disc}_{U}(G_{1})=0$ implies $h_{11}h_{12}=0$, cases we have already considered.
Finally, note that the numerator of $h_{23}$
% h23 =  y0^7*(-m^2*u^2+4*m^3*y0+2*m*u*y0-y0^2)/(m*u-y0)^4
is a divisor of the polynomial $y_0^{6}\op{Disc}_{X}(f)$
% discf = y0*(-m^2*u^2 + 4*m^3*y0 + 2*m*u*y0 - y0^2)*(m^3*u^6 - 4*m^2*u^5*y0 - 32*m^3*u^3*y0^2 + 5*m*u^4*y0^2 + 144*m^2*u^2*y0^3 - 2*u^3*y0^3 - 216*m*u*y0^4 + 108*y0^5))/(4*(m*u - y0)^6)
which implies that $h_{23}=0$ cannot be zero.
To reduce the genus of $\cal{C}$ we must therefore have $\op{Disc}_{U}(H_{2})=0$, i.e. $h_{21}=0$.
% h21=8*y0^2-4*(3*m+c2)*y0*(c1+2*c2*m+3*m^2)+(c1+2*c2*m+3*m^2)^3
We have
$$
h_{21}=\frac{4(2y_0-mu)y_0^3}{(y_0-mu)^2}
$$
and thus $u=\frac{2y_0}{m}$. In consequence, the polynomial $(a+bU^2)^2$ divides $H_{2}(U)$,
and the square-free part of the polynomial $H_{1}H_{2}$ is of degree 4.
Finally, the coefficients $(c_2,c_1,c_0)$ have become $(-2m+\frac{y_0}{m^2}, m^2, 0)$.
Summing up: if
$$
f(X)=X\left(X^2+\left(-2m+\frac{y_0}{m^2}\right)X+m^2\right),
$$
% X^3+(-2m+y0/m^2)X^2+m^2 X
then there is a non-constant rational map $\psi_{2}:\;\cal{C}_{2}\rightarrow \cal{W}_2$, where
\begin{equation}
\label{psi2}
\cal{C}_2:\;W^2=-aby_0(a+b U^2) (a(4a m^2 p_0^2-y_0) +8a b m^2 p_0 q_0 U +b(4b m^2 q_0^2-y_0) U^2)
\end{equation}
and
\begin{equation*}
\cal{W}_2:\;Y^2=f(ap^2+bq^2).
\end{equation*}
The second quadratic factor has discriminant $4aby_0(4 m^3-y_0)$, and if $y_0=4 m^3$, then $f(X)$ becomes
$X(X+m)^2$, disallowed. Thus $\cal{C}_2$ is of genus $1$. \\ \\
We can henceforth assume that $\op{Disc}(H_1) \neq 0$. \\ \\
{\bf Case IV}: suppose $h_{21}=0$. This equation defines a genus zero curve over the field $\Q(m,y_0)$ in the $(c_{1}, c_{2})$ plane with triple singular point at infinity $P_{0}=(-2m:1:0)$.
It has the parametrization:
\begin{equation}\label{h21zero}
c_{1}=-\frac{mu^3-2u(3m^2+u)y_0+8my_0^2}{2uy_0}, \qquad c_{2}=\frac{u^3-12muy_0+8y_0^2}{4uy_0}.
\end{equation}
% (c1,c2)=(  -(m*u^3 - 6*m^2*u*y0 - 2*u^2*y0 + 8*m*y0^2)/(2*u*y0), (u^3 - 12*m*u*y0 + 8*y0^2)/(4*u*y0)  )
where $u$ is a rational parameter.
With these values of $c_{1}, c_{2}$, then $H_1(U)=(a+bU^2) G_1(U)$, and
$H_2(U)=(a+bU^2)^2 G_2(U)$. By our assumptions, necessarily the discriminant of $G_{2}$
must vanish, equivalent to the vanishing of $h_{23}$.
% h23 = -c3*y0^5+2*m*(2*c2^2-c1*c3+26*c2*c3*m+95*c3^2*m^2)*y0^4-m*(c1*c2+18*c2^2*m+11*c1*c3*m+185*c2*c3*m^2+369*c3^2*m^3)*y0^3*Z+8*m^2*(c1*c2+4*c2^2*m+5*c1*c3*m+25*c2*c3*m^2+33*c3^2*m^3)*y0^2*Z^2-m^2*(c1+10*c2*m+27*c3*m^2)*y0*Z^4+m^3*Z^6
We have
$$
h_{23}=\frac{y_0^5(16my_0-u^2)}{u^2},
$$
and thus $h_{23}=0$ if and only if $16my_0-u^2=0$. Put $m y_0=v^2,u=4v$, and then
$G_2=-a b (q_0-p_0 U)^2 y_0^5/m$, and the square-free part of $H_1(U)H_2(U)$ is of
degree $4$.
The coefficients $(c_2,c_1,c_0)$ are now
\[ \left( \frac{2m^2+v}{2m}, \quad 3v-5m^2, \quad \frac{(3m^2-2v)(2m^2-v)}{2m} \right). \]
% (c2,c1,c0)=((2*m^2+v)/(2*m), 3*v-5*m^2, (3*m^2-2*v)*(2*m^2-v)/(2*m)  )
% deg4 = -2*a*b*(a+b*U^2)*(a*(2*a*m*p0^2-v) + 4*a*b*m*p0*q0*U + b*(2*m*b*q0^2-v)*U^2)
Summing up: if
\begin{align*}
f(X) = &\; X^3 + \left(\frac{2m^2+v}{2m}\right) X^2 - (5m^2-3v) X +  \frac{(3m^2-2v)(2m^2-v)}{2m}, \\
     = &\; \left(X-\frac{2m^2-v}{2m}\right)\left(X^2 + 2m X - 3m^2+2v\right)
\end{align*}
% f(X)=X^3+(2*m^2+v)/(2*m)*X^2+(3*v-5*m^2)*X+(3*m^2-2*v)*(2*m^2-v)/(2*m)
then there is a non-constant rational map $\psi_{3}:\;\cal{C}_{3}\rightarrow \cal{W}_3$,
where
\begin{align}\label{psi3}
&\cal{C}_{3}:\; W^2=2ab(a+bU^2)(a(2amp_0^2-v)+4abmp_0q_0U+b(2mbq_0^2-v)U^2), \\
&\cal{W}_3:\; Y^2=f(ap^2+bq^2).\notag
\end{align}
The second quadratic factor in $\cal{C}_3$ has discriminant $4abv(2m^2-v)$, and $2m^2=v$
implies $f(X)$ has a repeated factor. Thus $\cal{C}_3$ is of genus $1$. \\ \\
{\bf Case V}: suppose $h_{23}=0$. Due to the assumption $\op{Disc}_{U}(H_{1})\neq 0$, there are two possibilities:

\begin{enumerate}
\item $\op{Res}_{U}(H_{1},H_{2})=0$ and thus $g_{2}=h_{23}=0$ (the possibility $g_{1}=0$
has been covered since $g_1=h_{21}$).
\item $\op{Res}_{U}(H_{1},H_{2})\neq 0$, in which case the necessary condition for
$\cal{C}$ to have genus 1 is that $H_{2}$ is a square up to multiplication by an element
of $\Q(m, Y)$.
\end{enumerate}
In the first instance, the resultant of $g_2$ and $h_{23}$ with respect to $y_0$ is
\[  (4c_1-c_2^2)(c_1+2c_2m+3m^2)^{10}(16c_1^2-8c_1c_2^2+c_2^4-16c_1c_2m+6c_2^3m+3c_1m^2)^2, \]
which must be zero. \\
% (4c1-c2^2)*(c1+2*c2*m+3*m^2)^10*(16*c1^2-8*c1*c2^2+c2^4-16*c1*c2*m+6*c2^3*m+3*c1*m^2)^2
If $4c_1=c_2^2$, then either $y_0=m(c_2+2m)^2/4$, or $y_0=(2c_2+3m)(c_2+6m)^2/108$, and in
both cases, $f(X)$ is singular. \\
If $c_1+2c_2m+3m^2=0$, then $y_0=-4(c_2+3m)^3/27$, and $f(X)$ is singular.  \\
If $16c_1^2-8c_1c_2^2+c_2^4-16c_1c_2m+6c_2^3m+3c_1m^2=0$, then this curve of genus $0$
has parametrization
\[ (c_1,c_2)=\left (\frac{(m+u)^3(-7m+u)}{256m^2}, \quad \frac{(m-u)(m+u)}{8m} \right). \]

Then either $y_0=(3m-u)^2(5m+u)^3/1024m^2$, or $y_0=(2m-u)(11m-u)^2(5m+u)^3/(27648m^3)$,
and in both cases $f(X)$ is singular. \\ \\
In the second instance, we form the ideal generated by the coefficients of $d_6 x^6+...+d_0 - d_6(x^3 + p x^2+q x+r)^2$ and obtain the elimination ideal in $d_6,...,d_0$ which Magma tells us has a basis comprising 45 terms. Substitute into the defining polynomials of this basis the coefficients of $H_2$; there result $45$ polynomials in the variables $A(=a p_0^2)$, $B(=b q_0^2)$, (so $m=A+B$), $c_2$, $y_0$, $Z$, where we seek common zeros. When $Z=0$, the only solutions that arise correspond to $\op{Disc}(f)=0$, so henceforth we suppose $Z \neq 0$.

The value $c_2 = \frac{8 y_0^2-12 m y_0 Z+Z^3}{4 y_0 Z}$ leads only to
$Z^2=16 m y_0$, giving $y_0=\frac{Z^2}{16m}$ and
\[ (c_2,c_1,c_0)=\left(\frac{8m^2 + Z}{8m}, \quad \frac{-20m^2 + 3Z}{4}, \quad \frac{(6m^2 - Z)(8m^2 - Z)}{16m} \right), \]
with
\begin{align*}
f(X) & = X^3 + \frac{8m^2 + Z}{8m} X^2 + \frac{-20m^2 + 3Z}{4} X + \frac{(6m^2 - Z)(8m^2 - Z)}{16m} \\
     & = \left(X-\frac{8m^2-Z}{8m}\right)\left(X^2+2m X -\frac{6m^2-Z}{2}\right).
\end{align*}
There is now a non-constant rational map $\psi_4: \cal{C}_4 \rightarrow \cal{W}_4$ where
\begin{align}\label{psi4}
&\cal{C}_4: W^2 = 2 a b (a+b U^2) (8m(a p_0 + b q_0 U)^2 -(a+b U^2)Z), \\
&\cal{W}_4: Y^2=f(ap^2+bq^2).\notag
\end{align}
The second quadratic factor in $\cal{C}_4$ has discriminant $4ab(8m^2-Z)Z$, whose vanishing implies $\op{Disc}(f)$ is zero. Thus $\cal{C}_4$ has genus $1$. \\ \\
It remains only to treat the case where $c_2 \neq \frac{8 y_0^2-12 m y_0 Z+Z^3}{4 y_0 Z}$. The problem reduces to finding common zeros of the $45$ polynomials, which are of high degree in $A$, $B$, $g$, $y_0$, $Z$.
This is a tedious computation, requiring careful book-keeping of the factors of various resultants (to compute some of these resultants took more than 66 hours of CPU time). After much effort and polynomial manipulation, we discover that all the zeros correspond to the vanishing of $\op{Disc}(f)$, so that no further solutions arise. \\ \\

{\bf Case VI}: the final case. Suppose $g_{2}=0$. The equation defines a genus zero curve over the field $\Q(m, Y)$ in the $(c_{1}, c_{2})$ plane, with singular point at infinity $P_{0}=(-2m:1:0)$. There is the following parametrization:
\begin{equation}\label{g2zero}
c_{1}=\frac{(u-m)(u^3-3mu^2+2y_0)}{u^2}, \quad c_{2}=\frac{2u^3-3mu^2+y_0}{u^2}
\end{equation}
where $u$ is a rational parameter. The corresponding $f(X)$ is singular. \\ \\
As in the case of Theorems \ref{biquartic1}--\ref{biquartic3} we omit the computational details of checking that the image of the map $\psi_{i}$ constructed above for $i\in\{1,2,3,4\}$ is not finite provided that the set $\cal{C}_{i}(\Q)$ is not finite, and summarize the above case by case analysis in the following:
\begin{thm}
Let $i \in \{1,2,3,4\}$ and suppose that the set of rational points on the curve $\cal{C}_i$ at {\rm  (\ref{psi1}), (\ref{psi2}), (\ref{psi3}), (\ref{psi4})} is infinite. Then the set of rational points on the corresponding surface $\cal{W}_i$ is infinite too.
\end{thm}

\begin{exam}\label{exam1}
{\rm Take $(m,v)=(4,-16)$ in Case IV above. This leads to the elliptic curve
$$
E_{1}:\;Y^2=X^3+2X^2-128X+480=:f_{1}(X).
$$
%  (-6 + X)*(-80 + 8*X + X^2)
Then $E_{1}(\Q)\simeq \Z\times\op{Tor}(E_{1}(\Q))$, where the infinite part of $E_{1}(\Q)$ is generated by the point $G_{1}=(10,-20)$ and $\op{Tor}(E_{1}(\Q))\simeq \Z/4\Z$ is generated by the point $P_{1}=(4,8)$. We are interested in rational points on the surface
$$
\cal{W}_{1,b}:\;Y^2=f_{1}(p^2+bq^2).
$$
Take the point $P_{1}$ lying on $E_{1}$ and follow the method presented. In this case $a=1$ and $(p_{0}, q_{0}, Y_{0})=(2,0,8)$. Following the construction above gives the genus $1$ curve
$$
\cal{C}_{1,b}:\;W^2=2b(bU^2+1)(bU^2+3)
$$
and map
$$
\phi_{1,b}:\;\cal{C}_{1,b}\ni (U,W)\mapsto (p,q,Y)\in\cal{W}_{1,b},
$$
given by
\begin{equation*}
p=\frac{2b^2U^3+W}{bU(bU^2+1)},\quad q=\frac{W-2bU}{b(bU^2+1)},\quad Y=\frac{-W(4b^2 U^2+W^2)}{b^3 U^3(1+b U^2)^2}.
\end{equation*}
Suppose that the set $\cal{C}_{1,b}(\Q)$ is infinite. In this case the image of the set $\cal{C}_{1,b}$ by the  map $\phi_{1,b}$ cannot be finite (in the sense mentioned above). Indeed, if $p(U,W)^2+bq(U,W)^2=p_{0}^2+bq_{0}^2$ for some fixed $p_{0}, q_{0}$ and the point $(U, W)$ lies on $\cal{C}_{1,b}$ then $U$ necessarily satisfies the equation
$$
b^2(p_0^2+bq_0^2-6)U^4+b(p_0^2+b q_0^2-8)U^2-6=0.
$$
Because this equation has at most four rational roots and the set $\cal{C}_{1,b}(\Q)$ is infinite, then clearly the set $\phi_{1,b}(\cal{C}_{1,b}(\Q))$ is infinite. \\ \\
This proves, for example, that the set of rational points on the surface $\cal{W}_{1,1}$ is infinite. Indeed, for $b=1$ the curve $\cal{C}_{1,1}$ is birationally equivalent to the elliptic curve
$$
\cal{E}_{1,1}:\;y^2=x^3-x^2-4x-2
$$
(which is $2$-isogenous to $E_1$).
The rank of $\cal{E}_{1,1}$ is $1$, with $\op{Tor}(\cal{E}_{1,1})\simeq \Z/2\Z$ generated by the point $(-1,0)$. The infinite part of $ \cal{E}_{1,1}(\Q)$ is generated by the point $P_{1,1}=\left(-\frac{3}{4},-\frac{1}{8}\right)$. This implies that the set of rational points on $\cal{C}_{1,1}(\Q)$ is infinite and the set $\phi_{1,1}(\cal{C}_{1,1}(\Q))$ is an infinite subset of the set $\cal{W}_{1,1}(\Q)$.

\begin{rem}
{\rm One can check that the only values of $b$ such that the curve $\cal{C}_{1,b}$ has infinitely many rational points are elements of the set $B=\{1,2,3,6\}$ with corresponding points of infinite order $\{(1,4), (1/4,15/4), (1,12), (2,90)\}$. Indeed, this is a consequence of the fact that the (finite) set of square-free parts of the integer points lying on the curve $v^2=u(u+2)(u+6)$ (obtained from $\cal{C}_{1,b}$ by the substitution $(b,W)=(T/2u^2,v/2u)$ is equal to $B$. As a consequence, the set of those rational points lying on $E_{1}$ with $X$-coordinate represented by the quadratic form $p^2+bq^2$, where $b\in B$, is infinite.}
\end{rem}
}
\end{exam}

\begin{rem}
{\rm Note that we can use the same method for the question of representations of $X$-coordinates of rational points on elliptic curves given by the equation
$$
\cal{C}:\;Y^2+a_{1}XY+a_{3}Y=X^3+a_{2}X^2+a_{4}X+a_{6},
$$
where $a_{i}$ are given integers such that the curve $\cal{C}$ is non-singular. However, in this case we expect even more difficult computations to arise then the ones encountered in Case V of this section and thus it was decided to consider only the case of curves $\cal{C}$ with $a_{1}=a_{3}=0$.

}
\end{rem}

\section{Questions and problems}\label{sec6}

In this section we state some problems and questions which are natural in the light of results presented above.

Let $f(X)=\sum_{i=0}^{4}c_{i}X^{i}\in\Q[x]$ be without multiple roots and define $\cal{C}$ to be the elliptic quartic curve
$$
\cal{C}:\;Y^2=f(X),
$$
which we assume to have infinitely many rational points. Let
$$
G=\{Q_{1},\ldots, Q_{r}, T_{1},\ldots, T_{m}\}
$$
be a set of generators for the set $\cal{C}(\Q)$. Here $r$ is the rank of $\cal{C}$, $Q_{i}$ is of infinite order for $i=1,\ldots, r$, and $T_{j}$ is of finite order for $j=1,\ldots, m$. In particular, for each $Q\in\cal{C}(\Q)$ we have $Q=\sum_{i=1}^{r}[n_{i}]Q_{i}+\sum_{j=1}^{m}[\epsilon_{j}]T_{j}$ for some $n_{i}\in\Z$ and $\epsilon_{j}\in A_{j}$, where $A_{j}\subset \N$ is finite and $|A_{j}|=$ order of the point $T_{j}$ for $j=1,\ldots, m$.

Assume further that the set of rational points on the surface
$$
\cal{V}:\;(ap^2+bq^2)^2=f(X)
$$
is infinite, and define the set
$$
\cal{S}:=\left\{(p,q,{\bf n}, {\bf \epsilon})\in\Q\times\Q\times\N^{r}\times A_{1}\times\ldots \times A_{m}:\;ap^2+bq^2=\pi_{Y}\left(\sum_{i=1}^{r}[n_{i}]Q_{i}+\sum_{j=1}^{m}[\epsilon_{j}]T_{j}\right)\right\},
$$
where to shorten notation we write ${\bf n}=(n_{1},\ldots,n_{r}), {\bf \epsilon}=(\epsilon_{1},\ldots,\epsilon_{m})$, and where $\pi_{Y}:\;\Q^2\rightarrow \Q$ is projection onto the $Y$-coordinate. In particular, our assumptions imply that the set $\cal{S}$ is infinite. Moreover, we assume that the method presented in Section \ref{sec2} and Section \ref{sec3} works for $f$, in that the related curve $\cal{C}'$ has infinitely many rational points and there is a map $\phi:\cal{C}'\rightarrow \cal{C}$ given by $\phi=(\phi_{1},\phi_{2})$ with $\phi_{2}$ being of the form $ap^2+bq^2$ for some rational functions $p, q$.\\ \\
We ask the following:
%For simplicity, also assume that the rank of the corresponding elliptic curve $\cal{E}'$ is 1 and let $Q'$ generate the infinite part of the set %$\cal{C}'(\Q)$.
\begin{ques}\label{ques1}
What part of the set $\cal{S}$ is covered by the set $\phi(\cal{C}'(\Q))$?
\end{ques}
\begin{ques}\label{ques2}
Let $f\in\Q[X]$ be as above and suppose that the curve $\cal{C}$ has infinitely many rational points. Let $\mathcal{P}$ be the set of pairs $(a,b)$ of square-free integers such that the quadratic form $ap^2+bq^2$ is irreducible over $\Q$ and the set of rational points on the surface
$$
\cal{V}_{a,b}:\;(ap^2+bq^2)^2=f(X)
$$
is infinite. Can $\mathcal{P}$ be infinite?\\
Further, for any given pair $(a, b)\in\Z\setminus\{(\pm1,\mp1\}$ of square-free integers is it possible to find infinitely many quartic polynomials $f\in\Z[X]$ such that the surface $\cal{V}_{a,b}$ has infinitely many rational points? (We assume that the polynomial $f$ is not divisible by the fourth power of a non-zero integer $\neq 1$.)
\end{ques}

%\begin{prob}
%Find an example of a polynomial $f\in\Q[X]$ of degree 4 and without multiple roots such that there are $a, b\in\Z$ with the property that the set of %rational points on the surface
%\begin{equation*}
%\cal{V}:\;(ap^2+bq^2)^2=f(X)
%\end{equation*}
%is infinite, but the methods presented in Section \ref{sec2} and Section \ref{sec3} does not apply.
%\end{prob}

\begin{prob}\label{prob2}
Study analogs of Question~\ref{ques1} and Question~\ref{ques2} related to representation by a quadratic form of the $X$-coordinate of rational points on the elliptic curve $\cal{E}:\;Y^2=f(X)$, where $f$ is a cubic polynomial without multiple roots.
\end{prob}

Consider the curve $E:\;Y^2=f(X)$ with $f\in\Z[X]$, where $f$ is without multiple roots and $\op{deg}f=3,4$ depending on whether we ask about representations of $X$-coordinates  or $Y$-coordinates of rational points on $E$ by a binary quadratic form. Denote this set of represented coordinates by $S(E)$. It was noted that in each example of the application of our method throughout the paper, (part of) the set $S(E)$ is parameterized by a certain curve $\cal{C}$ of genus 1 {\it isogenous} to $E$. This suggests the following:

\begin{ques}\label{ques3}
Let $E:\;Y^2=f(X)$ where $f\in\Z[X]$ is without multiple roots and $\op{deg}f=3$ or 4. Suppose that the set $S(E)$ defined above is infinite. Does there exist a genus 1 curve $\cal{C}$ with infinitely many rational points and an isogeny $\phi:\;\cal{C}\rightarrow E$ such that $\pi(\phi(\cal{C}(\Q)))\subset S(E)$? (Here,  $\pi=\pi_{X}$ or $\pi=\pi_{Y}$ depending on whether we are asking about representability of $X$- or $Y$-coordinate.)
\end{ques}

Finally we ask the following:

\begin{ques}\label{ques4}
Does there exist a quartic polynomial $f_4$ (without multiple roots), and a binary cubic form $h(x,y) \in \Z[x,y]$, such that the surface
$$
\cal{S}_{1}:\;h(p,q)^2=f_{4}(X)
$$
has infinitely many non-trivial rational points?
Does there exist a cubic polynomial $f_3$ (without multiple roots), and a binary cubic form $h(x,y) \in \Z[x,y]$, such that the surface
$$
\cal{S}_{2}:\;Y^2=f_{3}(h(p,q))
$$
has infinitely many non-trivial rational points?
\end{ques}

\noindent
{\bf Acknowledgement}: All computations for this paper were carried out using Magma~\cite{Mag} and Mathematica~\cite{Wol}.

\vspace*{0.25in}

\noindent Andrew Bremner, School of Mathematics and Statistical Sciences, Arizona
State University, Tempe AZ 85287-1804, USA. e-mail: bremner@asu.edu

\bigskip

\noindent Maciej Ulas, Jagiellonian University, Institute of Mathematics,
{\L}ojasiewicza 6, 30-348 Krak\'ow, Poland. email:
Maciej.Ulas@im.uj.edu.pl

\end{document}